\documentclass[12pt,reqno,a4paper]{amsart}    

\usepackage[totalwidth=500pt, totalheight=700pt]{geometry}

\usepackage{amsmath, amssymb}
\usepackage{amsthm}

\usepackage{graphicx}                      
\newcommand{\color}[2][{}]{}        

\usepackage{bbm}         
\usepackage{mathrsfs} 
\renewcommand\mathcal\mathscr  

\usepackage{accents}   

\usepackage{slashed}    

\usepackage{stmaryrd}    

\theoremstyle{plain}            
\newtheorem{theorem}{Theorem}[section]

\newtheorem{lemma}[theorem]{Lemma}
\newtheorem{corollary}[theorem]{Corollary}

\theoremstyle{definition}       
\newtheorem{definition}[theorem]{Definition}

\theoremstyle{remark}           
\newtheorem{remark}[theorem]{Remark}

\newcommand{\Sec}[1]{Section~\ref{sec:#1}}

\newcommand{\Eq}[1]{Eq.~\eqref{eq:#1}}

\newcommand{\Thm}[1]{Theorem~\ref{thm:#1}}

\newcommand{\Lem}[1]{Lemma~\ref{lem:#1}}

\newcommand{\LemS}[2]{Lemmas~\ref{lem:#1}--\ref{lem:#2}}
\newcommand{\Cor}[1]{Corollary~\ref{cor:#1}}

\newcommand{\Rem}[1]{Remark~\ref{rem:#1}}
\newcommand{\Remenum}[2]{Remark~\ref{rem:#1}~(\ref{#2})}
\newcommand{\Def}[1]{Definition~\ref{def:#1}}

\newcommand{\Defenum}[2]{Definition~\ref{def:#1}~(\ref{#2})}

\numberwithin{equation}{section}

\DeclareMathOperator{\id}     {id}  

\DeclareMathOperator{\dom}    {dom}
\DeclareMathOperator{\ran}    {ran}

\newcommand{\spec}[2][{}]   {\sigma_{\mathrm{#1}}(#2)}

\def\XXint#1#2#3{{\setbox0=\hbox{$#1{#2#3}{\int}$}
     \vcenter{\hbox{$#2#3$}}\kern-.5\wd0}}

\newlength{\maxbreite}%
\newlength{\maxhoehe}%
\newlength{\maxtiefe}%

\newcommand{\stelldrueber}[3][0pt]{
  \settowidth{\maxbreite}{#3}%
  \settoheight{\maxhoehe}{#3}%
  \settodepth{\maxtiefe}{#2}%
  \addtolength{\maxhoehe}{\maxtiefe}%
  {\makebox[\maxbreite]{\raisebox{\maxhoehe}{\hspace{#1}#2}}%
  \makebox[0pt][r]{#3}}%
}

\newcommand{\overcirc}[1]       
{\stelldrueber[.45ex]{$\scriptscriptstyle\circ$}{${#1}$}}

\renewcommand{\phi}{\varphi}   
\newcommand{\im}{\mathrm i} 
\DeclareMathOperator{\dd}    {d\!} 

\newcommand{\wt}{\widetilde}           

\newcommand{\HS}{\mathcal H}           
\newcommand{\Bsymb} {\mathcal B}       

\newcommand{\Sobsymb} {\mathsf H}      
\newcommand{\Lsymb}    {\mathsf L}     

\newcommand{\Lin}[1]{\mathcal \Bsymb({#1})}

\newcommand{\Lsqr}[2][{}]{\Lsymb_2^{#1}({#2})} 
 
\newcommand{\Sob}[2][1]{\Sobsymb^{#1}({#2})} 
\newcommand{\Sobn}[2][1]{\ring \Sobsymb^{#1}({#2})}
 
\newcommand{\norm}[2][{}]{\|{#2}\|_{{#1}}}    
\newcommand{\normsqr}[2][{}]{\|{#2}\|^2_{#1}} 

\newcommand{\iprod}[3][{}]{\langle{#2},{#3}\rangle_{#1}}  

\newcommand{\set}[2]{\{ \, #1 \, | \, #2 \, \} } 
\newcommand{\bigset}[2]{\bigl\{ \, #1 \, \bigl|\bigr. \, #2 \, \bigr\} }

\newcommand{\map}[3]{ #1 \colon #2 \longrightarrow #3 } 
\newcommand{\partmap}[3]{ #1 \colon #2 \dashrightarrow #3 } 

\newcommand{\bd}  {\partial}                

\newcommand{\restr}[1]{{\restriction}_{#1}} 

\newcommand{\conj}[1]{\overline {{#1}}}       
\newcommand{\orth}{\bot}                    
\newcommand{\normder}{\partial_\mathrm{n}}  

\newcommand{\Neu}{{\mathrm N}}              
\newcommand{\Dir}{{\mathrm D}}              
 
\newcommand{\HSaux}{{\mathcal G}}

\usepackage{upgreek}

\newcommand{\ded}{\updelta}

\usepackage{amsbsy}  

\usepackage[dpi=600,PostScript=dvips]{diagrams}
\newarrow{Into}C--->
\newcommand{\dbar}[1]{\leavevmode\raise0.6ex\hbox{--}\kern-0.7em #1}

\newcommand{\HSn}{\ring \HS}

\newcommand{\dplus}{\mathop{\dot+}}

\newcommand{\lapl} [1][{}]{\Delta_{#1}} 
\newcommand{\laplD}[1][{}]{\Delta^\Dir_{#1}} 
\newcommand{\laplN}[1][{}]{\Delta^\Neu_{#1}} 

\newcommand{\laplX} [2][{}]{\Delta_{#1}^{#2}} 

\newcommand{\mc}{\mathcal}

\newcommand{\Graph} X  

\newcommand{\oplussplit}{\stackrel{\curlyveeuparrow}{\oplus}}

\newcommand{\de}   {\mathord {\mathrm d}}         
\DeclareMathOperator{\graph}{graph}

\declareslashed{}{\text{-}}{-0.13}{0.4}{D}

\newcommand{\embmap}[3]{ #1 \colon #2 \hookrightarrow #3 } 

\begin{document}
\title[First order operators and boundary triples]{First order
  operators and boundary triples}

\author{Olaf Post}      
\address{Institut f\"ur Mathematik,
         Humboldt-Universit\"at zu Berlin,
         Rudower Chaussee~25,
         12489 Berlin,
         Germany}
\email{post@math.hu-berlin.de}
\date{\today}




\begin{abstract}
  The aim of the present paper is to introduce a first order approach
  to the abstract concept of boundary triples for Laplace operators.
  Our main application is the Laplace operator on a manifold with
  boundary; a case in which the ordinary concept of boundary triples
  does not apply directly.  In our first order approach, we show that
  we can use the usual boundary operators also in the abstract Green's
  formula.  Another motivation for the first order approach is to give
  an intrinsic definition of the Dirichlet-to-Neumann map and
  intrinsic norms on the corresponding boundary spaces.  We also show
  how the first order boundary triples can be used to define a usual
  boundary triple leading to a Dirac operator.
  \hfill{\emph{\footnotesize In memoriam Vladimir A.~Geyler
      (1943-2007)}}
\end{abstract}

\maketitle

%
\section{Introduction}
\label{sec:intro}
%

The concept of boundary triples, originally introduced
in~\cite{vishik:63}, has successfully be applied to the theory of
self-adjoint extensions of symmetric operators, for example on quantum
graphs, singular perturbations or point interactions on manifolds (see
e.g.~\cite{bgp:pre06}). For a general treatment of boundary triples we
refer to~\cite{bgp:pre06,dhms:06} and the references therein.

Our main purpose here is not to characterise all self-adjoint
extensions of a given symmetric operator, but to show that the concept
of boundary triples can also be used in the PDE case, namely to
Laplacians on a manifold with boundary. The standard theory of
boundary triples does not directly apply in this case, since Green's
formula
\begin{equation*}
  \int_X {\lapl \conj f} g \dd x 
  - \int_X \conj f \lapl g \dd x 
  = \int_{\bd X} (\normder \conj f  g - \conj f \normder g)\restr{\bd X}
\end{equation*}
does not extend to $f,g$ in the maximal operator domain
\begin{equation*}
  \dom \laplX {\max} =
  \set{f \in \Lsqr X}{\laplX{\max} f \in \Lsqr X \text{ (distributional
    sense)}}
\end{equation*}
(cf.~\Rem{lapl.max.mfd} for details).  A solution to overcome this
problem is either to modify the boundary operators (restriction of the
function and the normal derivative onto $\bd X$) as e.g.\
in~\cite{bmnw:pre07,posilicano:pre07}, or to introduce the concept of
\emph{quasi} boundary triples as in~\cite{behrndt-langer:07} (cf.~also
the references therein for further treatments of boundary triples in
the PDE case).

Here, we use a different approach: we start with \emph{first order}
operators, namely the exterior derivative $\de$ taking functions
($0$-forms) to $1$-forms and its adjoint, the \emph{divergence
  operator} $\ded$, mapping $1$-forms into functions, since the first
order operator domains are simpler.  The Laplacian (on functions) is
then defined as $\lapl[0] := \ded \de$.  Certainly, in our approach we
do not cover all self-adjoint extensions of the minimal Laplacian.

The abstract approach also allows to define the Dirichlet-to-Neumann
map in an intrinsic manner, and also the norm of
$\HSaux^{1/2}=\Sob[1/2]{\bd X}$ is defined intrinsicly. This might be
a great advantage when dealing with parameter-depending manifolds, as
it is the case for graph-like manifolds (see
e.g.~\cite{exner-post:pre07b,post:06}). We will treat this question in
a forthcoming publication.  Our approach is related to the recent
works of Arlinskii~\cite{arlinskii:00},
Posilicano~\cite{posilicano:pre07} and Brown et~al.~\cite{bmnw:pre07},
where also a PDE example is treated in the context of boundary
triples.

To precise our idea of the first order approach we sketch the
construction here. The given data are\footnote{Here and in the sequel,
  $\partmap A {\HS_0}{\HS_1}$ denotes a partial map, i.e., a map (a
  linear operator) which is defined only on a subset $\dom A \subset
  \HS_0$.}
\begin{equation*}
  \HS_0, \quad \HS_1, \qquad
  \partmap \de {\HS_0}{\HS_1}, \quad \HS_0^1:=\dom \de,
\end{equation*}
where $\HS_p$ are Hilbert spaces (``$p$-forms''), and $\HS_0^1$
carries the graph norm. Guided by our main application (a manifold with
boundary), we call $\de$ an \emph{exterior derivative}.

A boundary map (of order $0$) is a bounded operator
\begin{equation*}
  \map {\gamma_0}{\HS_0^1} \HSaux, \qquad
  \HSaux^{1/2}:= \ran \gamma_0 
\end{equation*}
with dense range $\HSaux^{1/2} \subset \HSaux$, where
$\HSaux$ is another Hilbert space (usually over the boundary).

For these data, we define $\de_0 := \de$ restricted to $\HSn_0^1 :=
\ker \gamma_0$ and the \emph{divergence} operator $\ded := \de_0^*$
with domain $\HS_1^1 := \dom \ded$. Furthermore, we can define a
natural norm on $\HSaux^{1/2}$ using $\gamma_0$.

In addition, we have a boundary operator of order $1$, namely,
$\map{\gamma_1}{\HS_1^1}\HSaux$, with the same range $\ran \gamma_1 =
\ran \gamma_0 = \HSaux^{1/2}$. Moreover, an abstract Green's formula
is valid, i.e.,
\begin{equation*}
    \iprod {\de f_0} {g_1} - \iprod {f_0} {\ded g_1}
    = \iprod[\HSaux^{1/2}]{\gamma_0 f_0}{\gamma_1 g_1}.
\end{equation*}
Finally, $h_p = \beta_p^z \phi$ is the solution of the Dirichlet and
Neumann problem 
\begin{equation*}
  \lapl[p] h_p = z h_p,  \qquad \gamma_p h_p = \phi,
\end{equation*}
respectively; we call $\beta_p^z$ also a \emph{Krein $\Gamma$-field of
  order $p$}.

The \emph{Krein Q-function} is defined as
\begin{equation*}
  Q_0^z \phi := \gamma_1 \de \beta_0^z;
\end{equation*}
a bounded operator (on the boundary space $\HSaux^{1/2}$), closely
related to the usual Dirichlet-to-Neumann map $\Lambda(z)$ on a
manifold with boundary defined in~\Eq{dn}.

The main idea here is to consider the Laplacian $\lapl[0] f_0 := \ded
\de f_0$ on the space
\begin{equation*}
  \HS_0^2 := \dom \ded \de :=
  \bigset{f_0 \in \dom \de}{\de f_0 \in \dom\ded}
\end{equation*}
instead of the maximal domain $\dom \laplX[0] {\max} =\set{f_0 \in
  \HS_0} {\lapl[0] f_0 \in \HS_0}$. Although $\lapl[0]$ is not closed on
$\HS_0^2$, we can develop a suitable theory of boundary spaces. In
particular, for a bounded and self-adjoint operator $B$ in
$\HSaux^{1/2}$ we can show that the Laplacian $\lapl[0]$ restricted to
\begin{equation*}
  \dom \laplX[0]B 
    := \set {f_0 \in \HS_0^2}
            {\gamma_1 \de f_0 = B \gamma_0 f_0}
\end{equation*}
(Robin-type boundary conditions) is self-adjoint under a suitable
condition on the domain of the adjoint (fulfilled in our example of
the Laplacian on a manifold with boundary). Our main result is Krein's
resolvent formula for the resolvents of $\laplX[0]B$ and the Dirichlet
Laplacian $\laplD[0]$; and a spectral relation between the operators
$\laplX[0]B$ and $Q_0^z - B$, namely
\begin{equation*}
  \spec {\laplX[0] B} \setminus \spec {\laplD[0]}
  = \set{ z \notin \spec {\laplD[0]}} 
      { 0 \in \spec{Q_0^z - B}}.
\end{equation*}
(see \Thm{krein}). The main advantage of our approach is that it can
almost immediately be applied to the case of the Laplacian on a
manifold with boundary, using the standard boundary operator
(restriction of a function to the boundary and restriction of the
normal component of a $1$-form to the boundary).

The paper is organised as follows: In the next section, we develop the
concept of first order boundary triples. In \Sec{bd.triple} we show
how this concept fits into the usual theory of boundary triples.
\Sec{example} contains our motivating example, namely, the Laplacian
on a manifold with boundary.

\subsection*{Acknowledgements}
The author would like to dedicate this contribution to his dear
colleague at the Humboldt University, Vladimir A.~Geyler, who passed
away very suddenly.  The author acknowledges the financial support of
the Collaborative Research Center SFB 647 ``Space -- Time -- Matter.
Analytic and Geometric Structures'' in which Vladimir Geyler also was
a member. Finally, the author would like to thank Y.~Arlinskii for
drawing his attention to the article~\cite{arlinskii:00}, where a very
similar approach was used.

%
\section{First order approach}
\label{sec:1st.order}
%
In this section, we develop the concept of boundary triples for
operators acting in different Hilbert spaces; guided by our main
example of the exterior derivative on a manifold with boundary.
\begin{definition}
  \label{def:ext.der}
  Let $\HS=\HS_0 \oplus \HS_1$ and $\HSaux$ be Hilbert spaces.
  \begin{enumerate}
  \item Elements of $\HS_p$ are referred to as \emph{$p$-forms}.
  \item A partial map $\partmap {\de}{\HS_0}{\HS_1}$ is called an
    \emph{exterior derivative} if $\de$ is a closed map with dense
    domain $\HS_0^1 := \dom \de \subset \HS_0$. We endow $\HS_0^1$
    with the natural norm defined by
    \begin{equation*}
      \normsqr[\HS_0^1] {f_0} := \normsqr {f_0} + \normsqr{\de f_0}.
    \end{equation*}
  \item
    \label{bd}
    We call $\map {\gamma_0}{\HS_0^1}\HSaux$ a \emph{boundary map (of
      degree $0$)} associated to $\de$ iff $\gamma_0$ is bounded with
    dense image, and if $\HSn_0^1 := \ker \gamma_0 \subset
    \HS_0^1=\dom \de$ is dense in $\HS_0$. The auxiliary Hilbert space
    is also referred to as a \emph{boundary space}. We say that
    $\gamma_0$ is \emph{proper}, if $\gamma_0$ is not surjective,
    i.e., if $\HSaux^{1/2}:= \gamma_0(\HS_0^1) \subsetneq \HSaux$.
  \item The data $(\HS, \HSaux, \gamma_0)$ define a \emph{first order
      boundary triple} for the exterior derivative $\partmap \de
    {\HS_0}{\HS_1}$ if $\gamma_0$ a boundary map associated to $\de$.
  \end{enumerate}
\end{definition}
\begin{definition}
  \label{def:div}
  We set $\de_0 := \de \restr {\HSn_0^1}$, and call $\partmap
  {\ded:=\de_0^*} {\HS_1}{\HS_0}$ the \emph{divergence operator} with
  domain $\HS_1^1 := \dom \ded$ and $\HSn_1^1 := \dom \de^*$ (clearly,
  $\HSn_1^1 \subset \HS_1^1$, and $\HSn_1^1$ is dense in $\HS_1$ since
  $\de$ is densely defined).  We endow $\HS_1^1$ with the natural norm
  \begin{equation*}
    \normsqr[\HS_1^1] {f_1} 
    := \normsqr {f_1} + \normsqr{\ded f_1}.
  \end{equation*}
\end{definition}
\begin{definition}
  \label{def:lapl}
  \indent
  \begin{enumerate}
  \item We call $\Delta_0 := \ded \de$ the \emph{Laplacian of degree
      $0$} with domain
    \begin{equation*}
      \HS_0^2 := \dom \ded \de :=
      \bigset{f_0 \in \dom \de}{\de f_0 \in \dom\ded}
    \end{equation*}
    Similarly, $\Delta_1 := \de \ded$ is called the \emph{(maximal)
      Laplacian of degree $1$} with domain
    \begin{equation*}
      \HS_1^2 := \dom \de \ded :=
      \bigset{f_1 \in \dom \ded}{\ded f_1 \in \dom\de}.
    \end{equation*}
    We endow $\HS_p^2$ with the norms
    \begin{align*}
      \normsqr[\HS_0^2] {f_0} &:= \normsqr {f_0} + \normsqr{\de f_0} +
      \normsqr{\ded \de
        f_0},\\
      \normsqr[\HS_0^2] {f_1} &:= \normsqr {f_1} + \normsqr{\ded f_1}
      + \normsqr{\de \ded f_1}.
    \end{align*}
    We denote the eigenspaces by $\mc N_p^z := \ker (\Delta_p - z)
    \subset \HS_p^2$. For $z=-1$, we set $\mc N_p := \mc N_p^{-1}$.
  \item We call
    \begin{align*}
      \laplD[0] &:=\de_0^* \de_0, &
      \laplN[0] &:=\de^* \de,\\
      \laplD[1] &:=\de_0 \de_0^*, & \laplN[1] &:=\de \de^*
    \end{align*}
    with the appropriate domains the \emph{Dirichlet Laplacian of
      degree $p=0,1$} and the \emph{Neumann Laplacian of degree
      $p=0,1$}, respectively. Clearly, all these operators are
    self-adjoint and non-negative. We denote the corresponding
    resolvents by $R_p^\Dir:=(\laplD[p]+1)^{-1}$ and
    $R_p^\Neu:=(\laplN[p]+1)^{-1}$.
  \end{enumerate}
\end{definition}
The following diagram tries to illustrate the two scales of Hilbert
spaces associated to $\de$, $\de^*$ and $\de_0$, $\de_0^*=\ded$ (dotted
arrows). 
Note that only at order $1$, $0$ and $-1$, we have relations between
the two scales:
\begin{equation}
  \label{eq:sub.cov}
  \begin{diagram}
           &     &            &      &       &      & \HSn_0^1 \\
           &     &            &      &       & \ruDotsto^{(R_0^\Dir)^{1/2}} 
                        \ldDotsto(2,6)^{\hspace{22ex}\de_0}& \dInto\\
     \dots &\rTo & \HS_0^{-1} & 
       \rTo^{(R_0^\Neu)^{1/2}} & \HS_0 & \rTo & \HS_0^1 &  \rTo & \dots\\
           &     & \uOnto     & \ruDotsto \ldTo(2,4)^{\de \qquad} 
                                \luTo(2,4)^{\de^*\qquad \qquad \qquad \qquad} 
                                \ldDotsto(2,6)_{\hspace{22ex}\de_0} & 
                       & \ldTo(2,4)_{\hspace{15ex}\de}  
                         \luTo(2,4)_{\de^*\hspace{15ex}}
                          \luDotsto(2,2)^{\ded \qquad\qquad}& & \ruDotsto\\
           &     & \HSn_0^{-1}&      &       &      & \HS_1^1 \\
           &\ruDotsto& & \luDotsto_{\hspace{8ex}\ded}&  & \ruDotsto & \uInto\\
     \dots &\rTo & \HSn_1^{-1}& \rTo & \HS_1 & \rTo_{(R_1^\Neu)^{1/2}}
                                             & \HSn_1^1&\rTo & \dots\\
           &     & \dOnto     & \ruDotsto_{(R_1^\Dir)^{1/2}} 
                                              &       &       &       & \\
           &     & \HS_1^{-1} &      &       &       &       \\
  \end{diagram}
\end{equation}

\begin{remark}
  \label{rem:lapl}
  \indent
  \begin{enumerate}
  \item The spaces $\HS_p^2$ are complete, i.e., Hilbert spaces with
    their natural norms.
  \item
    \label{max}
    Note that $\Delta_p$ is a bounded operator on $\HS_p^2$. However,
    $\Delta_p$ with $\dom \Delta_p = \HS_p^2$ is \emph{not} closed.
    Although we call $\Delta_p$ the \emph{maximal} Laplacian, it is
    not the maximal operator $\Delta_p^{\max}$ in the usual sens
    (which is the operator closure of $\Delta_p$ with domain
    \begin{equation}
      \label{eq:lapl.max}
      \dom \laplX[p]{\max} :=
      \bigset{f_p \in \HS_p} {\lapl[p] f_p \in \HS_p}
    \end{equation}
    in the distributional sense). In general, $\HS_p^2 \subsetneq
    \dom \Delta_p^{\max}$. This observation is one of the motivations
    for our first order approach (see \Sec{example}).
  \end{enumerate}
\end{remark}

\begin{lemma}
  \label{lem:osum}
  We have $\HS_p^1 = \HSn_p^1 \oplus \mc N_p$ (orthogonal sum).
\end{lemma}
\begin{proof}
  Let $p=0$ and $f_0 \in \HS_0^1$. In this case, $f_0 \in
  (\HSn_0^1)^\orth$ is equivalent to
  \begin{equation}
    \label{eq:orel}
    0 = \iprod[\HS_0^1] {f_0} {g_0} =
        \iprod[\HS_0] {f_0} {g_0} + \iprod[\HS_1] {\de f_0} {\de g_0}, 
           \qquad \forall \, g_0 \in \HSn_0^1.
  \end{equation}
  However, by definition of the adjoint operator $\ded=\de_0^*$, we
  have $h_1 \in \dom \de_0^*$ iff there exists $h_0 \in \HS_0$ such
  that
  \begin{equation}
    \label{eq:def.adj}
        \iprod[\HS_0] {h_1} {\de_0 g_0} = \iprod[\HS] {h_0} {g_0} 
           \qquad \forall \, g_0 \in \HSn_0^1.
  \end{equation}
  Choosing $h_0=- f_0$, the orthogonality relation~\eqref{eq:orel}
  reads $h_1=\de f_0 \in \dom \de_0^*$ and $\de_0^* \de f_0 = -f_0$,
  i.e., $f_0 \in \mc N_0^z$. The argument for $p=1$ is similar.
\end{proof}

\begin{lemma}
  \label{lem:d.uni}
  The maps $\map \de {\mc N_0} {\mc N_1}$ and $\map \ded
  {\mc N_1} {\mc N_0}$ are unitary.
\end{lemma}
\begin{proof}
  If $f_0 \in \mc N_0$ then $\de \ded \de f_0 = - \de f_0$, i.e, $\de
  f_0 \in \mc N_1$. Similarly, $f_1 \in\mc N_1$ implies $\ded f_1 \in
  \mc N_0$. Furthermore, $-\ded \de f_0 = f_0$ and $\de (- \ded f_1) =
  f_1$ implies that $-\ded$ is the inverse of $\de$. Finally, $\de$ is
  an isometry because
  \begin{equation*}
    \normsqr[\HS_1^1]{\de f_0} 
    = \normsqr[\HS_1]{\de f_0} + \normsqr[\HS_0]{\ded \de f_0}
    = \normsqr[\HS_1]{\de f_0} + \normsqr[\HS_0]{f_0}
    = \normsqr[\HS_0^1]{f_0}.
  \end{equation*}
  Since $\de$ is surjective, it is therefore unitary with unitary
  inverse $-\ded$.
\end{proof}

\begin{lemma}
  \label{lem:g.inv}
  Assume that the boundary map $\gamma_0$ is proper (i.e.,
  $\HSaux^{1/2}=\ran\gamma_0 \subsetneq \HSaux$).  Define $\hat
  \gamma_0 := \gamma_0 \restr {\mc N_0}$, then $\hat \gamma_0$ is
  invertible and $\partmap{(\hat \gamma_0)^{-1}}\HSaux {\mc N_0}$ is an
  unbounded operator with domain $\dom (\hat \gamma_0)^{-1} =
  \HSaux^{1/2}$.  Furthermore, $(\hat \gamma_0)^{-1} \phi = h_0$ is the
  (unique) solution of the Dirichlet problem
  \begin{equation*}
    (\lapl[0]+1) h_0 = 0, \qquad \gamma_0 h_0 = \phi.
  \end{equation*}
\end{lemma}
\begin{proof}
  The operator $\hat \gamma_0$ is invertible since $(\ker
  \gamma_0)^\orth = (\HSn_0^1)^\orth=\mc N_0$ by \Lem{osum}. If $(\hat
  \gamma_0)^{-1}$ were be bounded, then $\hat \gamma_0$ would be a
  topological isomorphism of $\mc N_0$ and $\ran
  \gamma_0=\HSaux^{1/2}$, in particular, $\HSaux^{1/2}$ would be
  closed in $\HSaux$, and by the density, we would have
  $\HSaux^{1/2}=\HSaux$ --- a contradiction. The last assertion is an
  immediate consequence of \Lem{osum} and the definition of the
  inverse map $(\hat \gamma_0)^{-1}$.
\end{proof}

\begin{definition}
  \label{def:norm.g12}
  We endow $\HSaux^{1/2}$ with the norm
  \begin{equation*}
    \norm[\HSaux^{1/2}] \phi := \norm[\HS^1_0]{(\hat \gamma_0)^{-1}\phi}.
  \end{equation*}
\end{definition}
\begin{lemma}
  \label{lem:bd.map}
  Assume that the boundary map $\gamma_0$ is proper (i.e.,
  $\HSaux^{1/2}=\ran\gamma_0 \subsetneq \HSaux$), then the following
  assertions hold:
  \begin{enumerate}
  \item We have $\norm[\HSaux] \phi \le \norm{\gamma_0}
    \norm[\HSaux^{1/2}] \phi$ for $\phi \in \HSaux^{1/2}$.
  \item The operator $\gamma_0 \gamma_0^* \ge 0$ is invertible in
    $\HSaux$, and
      \begin{equation*}
        \Lambda := (\gamma_0 \gamma_0^*)^{-1}
        = ((\hat \gamma_0)^{-1})^*(\hat \gamma_0)^{-1} 
        \ge \frac 1 {\normsqr{\gamma_0}}.
      \end{equation*}
      We define the associated scale of Hilbert spaces by
      \begin{equation*}
        \HSaux^s := \dom \Lambda^s, \qquad
        \norm[\HSaux^s] \phi := \norm[\HSaux]{\Lambda^s \phi}
      \end{equation*}
      for $s \ge 0$ (and the dual with respect to
      $(\cdot,\cdot)_\HSaux$ for $s < 0$).
    \item The operator $\partmap{((\hat \gamma_0)^{-1})^*}{\mc
        N_0}\HSaux$ is unbounded with domain
      \begin{equation*}
        \dom ((\hat \gamma_0)^{-1})^* = 
        \set{f_0 \in \mc N_0}
            { \gamma_0 f_0 \in \dom \Lambda = \HSaux^1}.
      \end{equation*}
    \item The operator $\map{\gamma_0^*}{\HSaux}{\HS_0^1}$ is bounded,
      and $\gamma_0^* \phi=h_0$ is the unique Neumann solution, i.e.,
    \begin{equation*}
      (\lapl[0]+1) h_0 = 0, \qquad \gamma_0 h_0 \in \HSaux^{1}, 
                             \quad \Lambda \gamma_0 h_0 = \phi.
    \end{equation*}
  \end{enumerate}
\end{lemma}
\begin{remark}
  If $\gamma_0$ is not proper (i.e., if $\gamma_0$ is surjective,
  i.e., $\HSaux^{1/2}=\HSaux$), then all the above assertions remain
  valid except for the fact that $(\hat \gamma_0)^{-1}$, $((\hat
  \gamma_0)^{-1})^*$ and $\Lambda$ are \emph{bounded} operators.
\end{remark}
\begin{proof}
  The first assertion follows from
  \begin{equation*}
    \norm[\HSaux] \phi 
    =   \norm[\HSaux] {\hat \gamma_0 (\hat \gamma_0)^{-1} \phi}
    \le \norm{\gamma_0} \norm[\HS^1] {(\hat \gamma_0)^{-1} \phi}
    =   \norm{\gamma_0} \norm[\HSaux^{1/2}] \phi.
  \end{equation*}
  To prove the second, note that $\gamma_0 \gamma_0^* = \hat \gamma_0
  \hat \gamma_0^*$ is bijective and
  \begin{equation*}
    \iprod[\HSaux^{1/2}] \phi \phi
    = \iprod[\HS^1] {(\hat \gamma_0)^{-1}\phi} {(\hat \gamma_0)^{-1}\phi}
    = \iprod[\HSaux] \phi 
              {((\hat \gamma_0)^{-1})^* (\hat \gamma_0)^{-1}\phi}
    = \iprod[\HSaux] \phi 
              {\Lambda \phi}
  \end{equation*}
  if $(\hat \gamma_0)^{-1} \phi \in \dom ((\hat \gamma_0)^{-1})^*$, i.e,
  $\phi \in \dom \Lambda$. Furthermore, $\norm{\Lambda^{-1}}\le
  \normsqr{\gamma_0}$.

  The third assertion is a consequence of \Lem{g.inv}, and the domain
  characterisation can be seen readily. To prove the fourth assertion,
  take $h_0=\gamma_0^*\phi \in \ran \gamma_0^* \subset (\ker
  \gamma_0)^\orth= \mc N_0$; in this case
  \begin{equation*}
    \iprod[\HS_0^1] {h_0}{f_0} = \iprod[\HSaux]\phi {\gamma_0 f_0}
  \end{equation*}
  for all $f_0 \in \HS_0^1$. If $f_0 \in \mc N_0$, then
  \begin{equation*}
    \iprod[\HS_0^1] {h_0}{f_0} 
     = \iprod[\HSaux^{1/2}]{\gamma_0 h_0} {\gamma_0 f_0}
  \end{equation*}
  by definition of the norm on $\HSaux^{1/2}$. But the latter term
  equals $\iprod[\HSaux]{\Lambda\gamma_0 h_0} {\gamma_0 f_0}$ if
  $\gamma_0 h_0 \in \dom \Lambda$, and thus $\phi = \Lambda \gamma_0
  h_0$.
\end{proof}
\begin{remark}
  Note that $\HSaux^{-1/2}$ is the completion of $\HSaux$ with respect
  to the norm $\norm[\HSaux^{-1/2}]\phi = \norm[\HS_0^1]{\gamma_0
    \phi}$.
\end{remark}
\begin{definition}
  We define the \emph{boundary map of order $1$} as
  \begin{equation*}
    \map {\gamma_1}{\HS_1^1} \HSaux, \qquad
    \gamma_1 := -\gamma_0 \ded P_1
  \end{equation*}
  where $P_p$ is the orthogonal projection in $\HS_p^1$ onto the
  subspace $\mc N_p$.
\end{definition}

\begin{lemma}
  \label{lem:g1.ker}
  We have $\ker \gamma_1 = \HSn_1^1$, and $\map
  {\gamma_1}{\HS_1^1}{\HSaux}$ is bounded with norm
  $\norm{\gamma_1}=\norm{\gamma_0}$. Furthermore, $\ran
  \gamma_1=\HSaux^{1/2}$ and $\hat \gamma_1 := \gamma_1
  \restr{\mc N_1}$ is a unitary map from $\mc N_1$ onto
  $\HSaux^{1/2}$.
\end{lemma}
\begin{proof}
  If $f_1 \in \HSn_1^1=(\mc N_1)^\orth$, then $\gamma_1 f_1 = 0$
  since $P_1 f_1=0$. If $f_1 \in \mc N_1$, then $\gamma_1
  f_1 = - \gamma_0 \ded f_1=0$ iff $f_1=0$ since $\ded$ is unitary from
  $\mc N_1$ onto $\mc N_0 = (\ker \gamma_0)^\orth$.

  The boundedness follows from
  \begin{equation*}
    \norm[\HSaux]{\gamma_1 f_1}
    \le \norm[\HSaux]{\gamma_0 \ded P_1 f_1}
    \le \norm{\gamma_0}\norm[\HS_0^1]{\ded P_1 f_1}
    =   \norm{\gamma_0}\norm[\HS_1^1]{P_1 f_1}
    \le \norm{\gamma_0}\norm[\HS_1^1]{f_1}
  \end{equation*}
  by \Lem{d.uni}. Furthermore, for $f_0 \in \mc N_0$ set $f_1 := \de
  f_0$, then $\gamma_1 f_1 = - \gamma_0 \ded \de f_0 = \gamma_0 f_0$.
  In particular, $\norm{\gamma_1}=\norm{\gamma_0}$.  Finally,
  \begin{equation*}
    \norm[\HSaux^{1/2}]{\hat \gamma_1 f_1} =
    \norm[\HSaux^{1/2}]{\gamma_0 \ded f_1} =
    \norm[\HS_0^1]{\ded f_1} =
    \norm[\HS_1^1]{f_1}
  \end{equation*}
  for $f_1 \in \mc N_1$, since $\ded f_1 \in \mc N_0$ and by
  \Lem{d.uni}.

\end{proof}
\begin{lemma}
  \label{lem:green}
  The (abstract) Green's formula holds, namely,
  \begin{equation*}
    \iprod {\de f_0} {g_1} - 
    \iprod {f_0} {\ded g_1} =
    \iprod[\HSaux^{1/2}]{\gamma_0 f_0}{\gamma_1 g_1} =
    ({\gamma_0 f_0},{\wt \gamma_1 g_1})_\HSaux
  \end{equation*}
  where $\map{\wt \gamma_1 := \Lambda
    \gamma_1}{\HS_1^1}{\HSaux^{-1/2}}$.
\end{lemma}
\begin{proof}
  If $f_0 \in \HSn_0^1$, then the LHS vanishes since $\ded=\de_0^*$,
  and so is the RHS, since $\gamma_0 f_0=0$. Similarly, if $g_1 \in
  \HSn_1^1=\dom \de^*$, then the LHS vanishes since $\ded
  g_1=\de^*g_1$ and so is the RHS, because $\gamma_1 g_1=0$ by
  \Lem{g1.ker}. For $f_0 \in \mc N_0$ and $g_1 \in \mc N_1$, we have
  \begin{align*}
    \iprod {\de f_0} {g_1} - 
    \iprod {f_0} {\ded g_1} 
    &= -\iprod {\de f_0} {\de \ded g_1} - \iprod {f_0} {\ded g_1}\\
    &= -\iprod[\HS_0^1] {f_0} {\ded g_1} 
     = \iprod[\HSaux^{1/2}]{\gamma_0 f_0}{-\gamma_0 \ded g_1}
  \end{align*}
  by \Def{norm.g12}. The last assertion is obvious.
\end{proof}

\begin{corollary}
  \label{cor:green}
  We have
  \begin{align*}
    \iprod {\lapl[0] f_0} {g_0} -
        \iprod {f_0} {\lapl[0] g_0}
    &=  \iprod[\HSaux^{1/2}] {\gamma_0 f_0} {\gamma_1 \de g_0}
      - \iprod[\HSaux^{1/2}] {\gamma_1 \de f_0} {\gamma_0 g_0}\\
    &=  \iprod[\HSaux] {\gamma_0 f_0} {\wt \gamma_1 \de g_0}
     - \iprod[\HSaux] {\wt \gamma_1 \de f_0} {\gamma_0 g_0}
  \end{align*}
  for $f_0, g_0 \in \HS_0^2$.
\end{corollary}
The following lemma shows that $\Lambda=\Lambda(-1)$ is the
Dirichlet-to-Neumann map for the operator $\lapl[0] + 1$:
\begin{lemma}
  \label{lem:dn}
  For $\phi \in \HSaux^{1/2}$ and $h_0 := (\hat \gamma_0)^{-1} \phi$ we
  have
  \begin{equation*}
    \Lambda \phi = \wt \gamma_1 \de h_0.
  \end{equation*}
\end{lemma}
\begin{proof}
  By \Lem{green}, we have
  \begin{equation*}
    \iprod {\de f_0} {\de h_0} - \iprod {f_0} {\lapl[0] h_0}
    = ({\gamma_0 f_0},{\wt \gamma_1 \de h_0})_\HSaux.
  \end{equation*}
  On the other hand, we have
  \begin{multline*}
    \iprod {\de f_0} {\de h_0} - \iprod {f_0} {\lapl[0] h_0}
    = \iprod[\HS_0^1] {f_0}{h_0}\\
    = \iprod[\HSaux^{1/2}] {\gamma_0 f_0}{\gamma_0 h_0}
    = \iprod[\HSaux^{1/2}] {\gamma_0 f_0} \phi
    = ({\gamma_0 f_0},\Lambda \phi)_\HSaux.
  \end{multline*}
  for $f_0, h_0 \in \mc N_0$.
\end{proof}
\begin{remark}
  \label{rem:dn}
  The map $\wt \gamma_1$ is indeed the boundary map occuring in the
  applications (see \Sec{example}). Namely, the Green's formula is usually
  formulated with a boundary integral given as an inner product of
  $\HSaux$ rather than $\HSaux^{1/2}$. In particular, $\wt \gamma_1
  \de h_0$ is the ``normal derivative at the boundary'' (in the case
  of a manifold with boundary). 
\end{remark}

The boundary maps are also bounded as maps with target space
$\HSaux^{1/2}$:
\begin{lemma}
  \label{lem:bd.g12}
  The operators $\map {\gamma_p}{\HS_p^1}{\HSaux^{1/2}}$ are bounded
  with norm bounded by~$1$.
  \end{lemma}
\begin{proof}
  For $p=0$, we have
  \begin{equation*}
    \norm[\HSaux^{1/2}]{\gamma_0 f_0}
    = \norm[\HS_0^1]{(\hat \gamma_0)^{-1}\gamma_0 f_0}
    = \norm[\HS_0^1]{(\hat \gamma_0)^{-1}\gamma_0 P_0 f_0}
    = \norm[\HS_0^1]{P_0 f_0}
    \le  \norm[\HS_0^1]{f_0},
  \end{equation*}
  since $\gamma_0 f_0 = \gamma_0 P_0 f_0$.  For $p=1$, we obtain
  \begin{equation*}
    \norm[\HSaux^{1/2}]{\gamma_1 f_1}
    = \norm[\HS_0^1]{(\hat \gamma_0)^{-1}\gamma_0 \ded P_1 f_1}
    = \norm[\HS_0^1]{\ded P_1 f_1}
    = \norm[\HS_1^1]{P_1 f_1}
    \le  \norm[\HS_1^1]{f_1}
  \end{equation*}
  using \LemS{d.uni}{g.inv}.
\end{proof}

In order to define the Dirichlet-to-Neumann map also for other
resolvent values~$z$, we need to provide results similar to those in
\LemS{osum}{g.inv} for general $z$. Write
\begin{equation}
  \label{eq:def.sigma}
  \Sigma_0 := \spec {\laplD[0]}, \qquad
  \Sigma_1 := \spec {\laplN[1]}.
\end{equation}
\begin{lemma}
  \label{lem:osum.z}
  For $z \notin \Sigma_p$, we have $\HS_p^1 = \HSn_p^1 \dplus \mc
  N_p^z$ (topological direct sum). In particular, $\hat \gamma_p^z :=
  \gamma_p \restr {\mc N_p^z}$ is a topological isomorphism from $\mc
  N_p^z$ onto $\HSaux^{1/2}$.
\end{lemma}
\begin{proof}
  For $z \notin \spec {\laplD[0]}$, we define
  \begin{equation*}
    \map{P_0^z :=  1 - \iota_0 (\laplD[0]-z)^{-1}  (\lapl[0] - z)}
       {\HS_0^1} {\HS_0^1}
  \end{equation*}
  where
  \begin{align*}
    \map{\lapl[0]=\ded \de}{\HS_0^1}{\HSn_0^{-1}}, \qquad
    \map{(\laplD[0]-z)^{-1}=(\ded \de_0-z)^{-1}}{\HSn_0^{-1}}{\HSn_0^1}
  \end{align*}
  and $\embmap{\iota_0}{\HSn_0^1}{\HS_0^1}$. A simple calculation
  shows that $(1-P_0^z)^2 = (1-P_0^z)$, i.e., $1-P_0^z$ and therefore
  $P_0^z$ are projections. Furthermore, $f_0=P_0^z f_0$ is equivalent
  to $\lapl[0] f_0= z f_0$. In order to show that $f_0 = P_0^z f_0 \in
  \mc N_0^z$ let us first show that $f_0 \in \HS_0^2$, i.e., that
  $h_1:=\de f_0 \in \HS_1^1 = \dom \ded$. To this end, recall the
  definition of the domain $\dom \ded=\dom \de_0^*$
  in~\eqref{eq:def.adj}. We have here
  \begin{equation*}
    \iprod {\de f_0}{\de_0 g_0}
    = \iprod {\ded \de f_0}{g_0}
    = \iprod {z f_0}{g_0}
  \end{equation*}
  by \Lem{green} (note that $\gamma_0 g_0=0$) and the fact that $\ded
  \de f_0 = z f_0$; we can choose $h_0=z f_0$ and therefore $f_0 \in
  \HS_0^2$. A straightforward calculation shows now that
  $(\lapl[0]-z)f_0=0$, and finally, $f_0 \in \mc N_0^z$.

  By the definition of $P_0^z$, it is also clear that $\ran (1-P_0^z)
  \subset \HSn_0^1$, and therefore $\HS_0^1$ splits into the direct
  sum. The direct sum is also a topological sum, since $1-P_0^z$ and
  $P_0^z$ are bounded maps. Therefore $f_0 \mapsto ((1-P_0^z)f_0,
  P_0^z f_0)$ is a bounded bijection, and also a topological
  isomorphism.  The argument for $1$-forms is similar, using
  \begin{equation*}
    \map{P_1^z :=  1 - \iota_1 (\laplN[1]-z)^{-1}  (\lapl[1] - z)}
       {\HS_1^1} {\HS_1^1}
  \end{equation*}
  where
  \begin{align*}
    \map{\lapl[1]=\de \ded}{\HS_1^1}{\HSn_1^{-1}}, \qquad
    \map{(\laplN[1]-z)^{-1}=(\de \de^*-z)^{-1}}{\HSn_1^{-1}}{\HSn_1^1}
  \end{align*}
  and $\embmap{\iota_1}{\HSn_1^1}{\HS_1^1}$.

  For the last assertion, note that $\ker \gamma_p = \HSn_p^1$ and
  that $\ran \gamma_p = \HSaux^{1/2}$ (see \Lem{g1.ker}); in
  particular, $\hat \gamma_p^z$ is bijective. Furthermore, $\hat
  \gamma_p^z$ is bounded as restriction of the bounded map
  $\map{\gamma_p}{\HS_p^1}{\HSaux^{1/2}}$ (cf.~\Lem{bd.g12}), and
  therefore, $\hat \gamma_p^z$ is a topological isomorphism.
\end{proof}
\begin{lemma}
  \label{lem:d.uni.z}
  For $z \neq 0$, the maps $\map \de {\mc N_0^z} {\mc N_1^z}$
  and $\map \ded {\mc N_1^z} {\mc N_0^z}$ are topological
  isomorphisms.
\end{lemma}
\begin{proof}
  If $f_0 \in \mc N_0^z$ then $\de \ded \de f_0 = z \de f_0$, i.e,
  $\de f_0 \in \mc N_1^z$. Similarly, $f_1 \in\mc N_1^z$ implies $\ded
  f_1 \in \mc N_0^z$. Furthermore, $\frac 1 z \ded \de f_0 = f_0$ and
  $\de (\frac 1 z \ded f_1) = f_1$ implies that $\frac 1 z \ded$ is
  the inverse of $\de$. Finally, $\de$ is bounded on $\mc N_0^z$,
  since
  \begin{equation*}
    \normsqr[\HS_1^1]{\de f_0} 
    = \normsqr[\HS_1]{\de f_0} + \normsqr[\HS_0]{\ded \de f_0}
    = \normsqr[\HS_1]{\de f_0} + |z|^2 \normsqr[\HS_0]{f_0}
    \le (1+ |z|^2)\normsqr[\HS_0^1]{f_0}
  \end{equation*}
  and therefore a topological isomorphism. The assertion for $\ded$
  follows similarly.
\end{proof}

\begin{definition}
  \label{def:krein.g}
  We call $z \mapsto \beta_0^z:=(\hat \gamma_0^z)^{-1}$, $z \notin
  \Sigma_0$ the \emph{Dirichlet solution map} or the \emph{Krein
    $\Gamma$-field of order $0$} associated to the first order
  boundary triple $(\HS,\HSaux, \gamma_0)$. Similarly, we call $z
  \mapsto \beta_1^z:=(\hat \gamma_1^z)^{-1}$, $z \notin \Sigma_1$ the
  \emph{Neumann solution map} or the \emph{Krein $\Gamma$-field of
    order $1$}.
\end{definition}

\begin{remark}
  \label{rem.beta}
  \indent
  \begin{enumerate}
  \item We prefer to use the symbol $\beta$ instead of $\gamma$ for
    the Krein $\Gamma$-field in order to avoid confusion with our
    boundary maps $\gamma_p$.

  \item The maps $\map {\beta_p^z}{\HSaux^{1/2}}{\mc N_p^z \subset
      \HS_p^1}$ are topological isomorphisms, since the inverses $\hat
    \gamma_p^z$ are.
  \item The names ``Dirichlet/Neumann solution map'' are due to the
    following fact: The $p$-form $h_p := \beta_p^z \phi$ is the
    solution of $(\lapl[p]-z)h_p=0$, and $\gamma_p h_p = \phi$. For
    $p=0$, this is the solution of the ``Dirichlet problem''
    ($\gamma_0 h_0$ prescribed), and for $p=1$, the solution of the
    ``Neumann problem'' ($\gamma_1 h_1$ prescribed). We will see in
    \Lem{krein.g3} that the Krein $\Gamma$-fields are related to a
    Krein $\Gamma$-field in the sense of an ordinary boundary triple.
  \item
    \label{adj}
    The map $\map {\beta_0^z}{\HSaux^{1/2}} {\HS_0^1}$ regarded as
    an operator $\map {\beta_0^z}{\HSaux^{1/2}} {\HS_0}$ into $\HS_0$ is
    bounded, as well as its adjoint, denoted by $\map{(\beta_0^{\conj
        z})^*}{\HS_0} {\HSaux^{1/2}}$.
  \end{enumerate}
\end{remark}

\begin{lemma}
  \label{lem:beta.adj}
  We have $\gamma_1 \de f_0 = (\beta_0^{\conj z})^* (\lapl[0]-z) f_0$
  for $f_0 \in \dom \laplD[0] = \HS_0^2 \cap \HSn_0^1$ where
  $(\beta_0^{\conj z})^*$ is the adjoint of $\beta_0^{\conj z}$ as
  operator $\map {\beta_0^z}{\HSaux^{1/2}} {\HS_0}$. Furthermore,
  $\ran(\beta_0^{\conj z})^*=\HSaux^{1/2}$.
\end{lemma}
\begin{proof}
  The assertion follows from (see
  also~\cite[Thm.~1.23~(2d)]{bgp:pre06})
  \begin{multline*}
    \iprod[\HSaux^{1/2}] \phi {(\beta_0^{\conj z})^*(\lapl[0]-z) f_0}
    = \iprod[\HS] {\beta_0^z \phi} {(\lapl[0]-z) f_0}\\
    = \iprod[\HS] {(\lapl[0]-z) \beta_0^z \phi} { f_0}
    + \iprod[\HSaux^{1/2}] {\gamma_0 \beta_0^z \phi} {\gamma_1 \de f_0}
    - \iprod[\HSaux^{1/2}] {\gamma_1 \de \beta_0^z \phi} {\gamma_0
         f_0}\\
    =\iprod[\HSaux^{1/2}] \phi {\gamma_1 \de f_0}
  \end{multline*}
  by \Cor{green} for the second equality. As far as the third equality
  is concerned, note that the first term vanishes since $\beta_0^z
  \phi$ solves the eigenvalue equation; the same holds for the third
  term since $\gamma_0 f_0=0$ for $f_0 \in \HSn_0^1$. For the second
  term, we have $\gamma_0 \beta_0^z \phi = \phi$ by the definition of
  $\beta_0^z$. The last assertion follows from $\ran(\beta_0^{\conj
    z})^*=(\ker \beta_0^{\conj z})^\orth$ and from the fact that
  $\map{\beta_0^{\conj z}}{\HSaux^{1/2}}{\HS_0}$ is injective.
\end{proof}

We can now define the Dirichlet-to-Neumann map and a closely related
map for arbitrary resolvent values~$z$:
\begin{definition}
  \label{def:dn.z}
  The \emph{Krein Q-function} associated to the first order boundary
  triple $(\HS, \HSaux, \gamma_0)$ is the map
  \begin{equation*}
    z \mapsto Q_0^z 
      := \gamma_1 \de (\hat \gamma_0^z)^{-1}=\gamma_1 \de \beta_0^z, 
          \qquad z \notin \Sigma_0=\spec{\laplD[0]}.
  \end{equation*}
  For $z \notin \Sigma_0$, the \emph{abstract Dirichlet-to-Neumann map
    at $z$} is defined by
  \begin{equation*}
    \map{\Lambda(z) 
    := \Lambda Q_0^z 
     = \Lambda \gamma_1 \de \beta_0^z
     = \wt \gamma_1 \de \beta_0^z}{\HSaux^{1/2}}{\HSaux^{-1/2}}.
\end{equation*}
\end{definition}

\pagebreak
\begin{remark}
  \indent
  \begin{enumerate}
  \item We shall see in \Sec{bd.triple} that $Q_0^z$ is indeed a Krein
    Q-function for an ordinary boundary triple. Note that
    $\map{Q_0^z}{\HSaux^{1/2}}{\HSaux^{1/2}}$ is a bounded map
    (cf.~\LemS{bd.g12}{d.uni.z}). In addition, we have
    \begin{equation*}
      Q_0^{-1} 
      = \gamma_1 \de (\hat \gamma_0)^{-1}
      = - \gamma_0 \ded P_1 \de (\hat \gamma_0)^{-1}
      = \gamma_0 (\hat \gamma_0)^{-1} = \id_{\HSaux^{1/2}}
    \end{equation*}
    at $z=-1$.
  \item Note that $\Lambda(z)$ is indeed the Dirichlet-to-Neumann map:
    We solve the Dirichlet problem $h_0 = \beta_0^z \phi$, i.e,
    \begin{equation*}
      \lapl[0] h_0 = z h_0, \qquad \gamma_0 h_0 = \phi;
    \end{equation*}
    and the Dirichlet-to-Neumann map is the ``normal derivative at the
    boundary'' of $h_0$ (cf.~\Rem{dn}), i.e., $\Lambda(z) \phi = \wt
    \gamma_1 \de h_0$.
  \end{enumerate}
\end{remark}

Let us now define self-adjoint restrictions of $\lapl[0]$.
\begin{definition}
  \label{def:lapl.sa}
  Let  $B$ be a bounded operator in $\HSaux^{1/2}$. We set
  \begin{gather*}
    \dom \laplX[0]B 
    := \set {f_0 \in \HS_0^2}
            {\gamma_1 \de f_0 = B \gamma_0 f_0}\\
    \dom \laplX[1]B 
    := \set {f_1 \in \HS_1^2}
            {\gamma_1 f_1 = B \gamma_0 \ded f_1}
  \end{gather*}
  and denote by $\laplX[p]B$ the restriction of $\lapl[p]$ onto $\dom
  \laplX[p]B$.
\end{definition}

\begin{lemma}
  \label{lem:lapl.sa}
  Assume that $\dom (\laplX[0]B)^* \subset \HS_0^1$, then the operator
  $\laplX[0]B$ is self-adjoint iff $B$ is self-adjoint in $\mc
  G^{1/2}$.
\end{lemma}
\begin{remark}
  The domain condition does not seem to follow from abstract
  (``soft'') arguments; in our manifold example, it follows from
  elliptic regularity (``hard'' arguments). Note that in general,
  $\dom \laplX[0]{\max}$ defined in~\eqref{eq:lapl.max} is even not a
  subset of $\HS_0^1$ (see \Remenum{lapl}{max} and
  \Rem{lapl.max.mfd}).
\end{remark}
\begin{proof}
  The graph of the operator $(\laplX[0]B)^*$ is given as
  \begin{multline*}
    \graph (\laplX[0]B)^* =
    \bigset{(f_0, \lapl[0] f_0)}
       {f_0 \in \dom \laplX[0]{\max}, \\
        \forall \, g_0 \in \dom \laplX[0]B \colon
        \iprod {\laplX[0]{\max} f_0} {g_0} =
        \iprod {f_0} {\laplX[0]{\max} g_0}}
    \subset \HS_0^1 \times \HS_0,
  \end{multline*}
  and the latter inclusion holds by our assumption on the domain of
  the adjoint. In particular, $f_0, g_0 \in \HS_0^2$ and we can apply
  \Cor{green}, namely,
  \begin{align*}
    \iprod {\laplX[0]{\max} f_0} {g_0} -
        \iprod {f_0} {\laplX[0]{\max} g_0} 
    &=  \iprod[\HSaux^{1/2}] {\gamma_0 f_0} {\gamma_1 \de g_0}
       -\iprod[\HSaux^{1/2}] {\gamma_1 \de f_0} {\gamma_0 g_0}\\
    &=  \iprod[\HSaux^{1/2}] {\gamma_0 f_0} {B\gamma_0 g_0}
       -\iprod[\HSaux^{1/2}] {B \gamma_0 f_0} {\gamma_0 g_0},
  \end{align*}
  and the latter equality follows from $f_0, g_0 \in \dom \laplX[0]B$.
  The assertion is now obvious.
\end{proof}

The self-adjointness of $B$ in $\HSaux^{1/2}$ can be shown as follows:
\begin{lemma}
  \label{lem:b.12}
  Let $\wt B$ be a bounded and self-adjoint operator on $\HSaux$. In
  this case, $B:= \Lambda^{-1} \wt B$ is bounded and self-adjoint as
  operator on $\HSaux^{1/2}$.
\end{lemma}
\begin{proof}
  We have
  \begin{equation*}
    \norm[\Lin {\HSaux^{1/2}}] B
    =\norm[\Lin {\HSaux}] {\Lambda^{1/2}B \Lambda^{-1/2}}
    =\norm[\Lin {\HSaux}] {\Lambda^{-1/2}\wt B \Lambda^{-1/2}}
    \le \norm[\Lin {\HSaux}] {\Lambda^{-1}}
             \norm[\Lin {\HSaux}] {\wt B},
  \end{equation*}
  so that $B$ is bounded on $\HSaux^{1/2}$, and
  \begin{equation*}
    \iprod[\HSaux^{1/2}]{B \phi} \psi
    = \iprod[\HSaux]{\Lambda^{1/2}B \phi} {\Lambda^{1/2}\psi} 
    = \iprod[\HSaux]{\Lambda^{-1/2}\wt B \phi} {\Lambda^{1/2}\psi} 
    = \iprod[\HSaux]{\wt B \phi} \psi 
  \end{equation*}
  and the similar symmetric expression shows the self-adjointness.
\end{proof}

We can now formulate our main result.  For brevity, we restrict
ourselves here to $0$-forms. Similar results hold also for $1$-forms.
\begin{theorem}
  \label{thm:krein}
  Let $B$ be a self-adjoint and bounded operator in $\HSaux^{1/2}$,
  $\laplD[0]$ the self-adjoint Laplacian with Dirichlet boundary
  conditions (cf.~\Def{lapl}) and $\laplX[0] B$ the self-adjoint
  restriction of the Laplacian (cf.\ \Def{lapl.sa}). Assume that $\dom
  (\laplX[0]B)^* \subset \HS_0^1$.
  \begin{enumerate}
  \item 
    \label{kernel} 
    For $z \notin \spec{\laplD[0]}$ we have
    $\ker (\laplX[0] B - z) = \beta_0^z \ker (Q_0^z - B)$.
  \item
    \label{krein}
    For $z \notin \spec{\laplX[0] B} \cup \spec{\laplD[0]}$ we have $0
    \notin \spec{Q_0^z - B}$ and Krein's formula
    \begin{equation*}
        (\laplD[0] - z)^{-1} - (\laplX[0]B - z)^{-1}
        = \beta_0^z (Q_0^z - B)^{-1} (\beta_0^{\conj z})^*
    \end{equation*}
    is valid, where $(\beta_0^{\conj z})^*$ is the adjoint of
    $\beta_0^{\conj z}$ as operator $\map {\beta_0^{\conj
        z}}{\HSaux^{1/2}} {\HS_0}$.
  \item We have
    \begin{equation*}
          \spec {\laplX[0] B} \setminus \spec {\laplD[0]}
          = \set{ z \notin \spec {\laplD[0]}} 
                { 0 \in \spec{Q_0^z - B}}.
    \end{equation*}
  \end{enumerate}
\end{theorem}
\begin{proof}
  The proof is again closely related to the proof for ordinary
  boundary triples (cf.~\cite[Thm.~1.29]{bgp:pre06}). For the first
  assertion, take $\phi \in \ker(Q_0^z - B)$ and set $f_0 = \beta_0^z
  \phi$. By the definition of the solution map $\beta_0^z$, we have
  $(\lapl[0] - z) f_0 = 0$ and $\gamma_0 f_0 = \phi$. Furthermore,
  $Q_0^z \phi = B \phi$ is equivalent to $\gamma_1 \de f_0 = B
  \gamma_0 f_0$ by the definition of $Q_0^z$. However, the last
  equation shows that $f_0 \in \dom \laplX[0] B$, i.e., $f_0 \in \ker
  (\laplX[0] B - z)$. The opposite inclusion follows similarly.

  To prove the second assertion, take $h_0 \in \HS_0$ and $f_0 :=
  (\laplX[0] B - z)^{-1} h_0 \in \dom \laplX[0] B$. By \Lem{osum.z} we
  can decompose $f_0 = f_0^z \dplus g_0^z \in \HSn_0^1 \dplus \mc
  N_0^z$. Since $f_0, g_0^z \in \HS_0^2$ we also have $f_0^z \in
  \HS_0^2$ and
  \begin{equation*}
    h_0 = (\laplX[0] B - z) f_0
    = (\lapl[0] - z) f_0
    = (\lapl[0] - z) f_0^z
    = (\laplD[0] - z) f_0^z,
  \end{equation*}
  i.e., $f_0^z = (\laplD[0] - z)^{-1} h_0$.
  Furthermore, $\gamma_0 f_0^z = 0$, therefore $\gamma_0 f_0 =
  \gamma_0 g_0^z$, i.e., $g_0^z = \beta_0^z \gamma_0 f_0$ and we have
  \begin{equation}
    \label{eq:krein.calc}
    (\laplX[0] B - z)^{-1} h_0
    = f_0 = f_0^z + g_0^z
    = (\laplD[0] - z)^{-1} h_0 + \beta_0^z \gamma_0 f_0.
  \end{equation}
  Now we apply $\gamma_1 \de$ to the decomposition of $f_0 \in \dom
  \laplX[0]B$ and obtain
  \begin{multline*}
    B \gamma_0 f_0 
    = \gamma_1 \de f_0 
    = \gamma_1 \de f_0^z  + \gamma_1 \de \beta_0^z \gamma_0 f_0\\
    = (\beta_0^{\conj z})^* (\lapl[0]-z) f_0^z  + Q_0^z \gamma_0 f_0
    = (\beta_0^{\conj z})^* h_0 + Q_0^z \gamma_0 f_0.
  \end{multline*}
  using the definition of $Q_0^z$ (cf.~\Def{dn.z}) and \Lem{beta.adj} for the
  third equality. In particular,
  \begin{equation}
    \label{eq:krein.calc2}
    (Q_0^z - B) \gamma_0 f_0 = (\beta_0^{\conj z})^* h_0,
  \end{equation}
  and the RHS covers the entire space $\HSaux^{1/2}$ since $h_0$
  covers $\HS_0$ (see again \Lem{beta.adj}). In particular, $(Q_0^z -
  B)$ is surjective.  By~\eqref{kernel}, this operator is also
  injective, i.e., $0 \notin \spec{Q_0^z - B}$.  Krein's formula now
  follows from~\eqref{eq:krein.calc}--\eqref{eq:krein.calc2}. The last
  assertion is a consequence of~\eqref{krein}.
\end{proof}

Returning to the original boundary space $\mc G$ and the
Dirichlet-to-Neumann map $\Lambda(z)=\Lambda Q_0^z$ --- regarded as an
unbounded operator in $\HSaux$ ---, we obtain the following result:
\begin{theorem}
  \label{thm:krein.dn}
  Let $\wt B$ be a self-adjoint and bounded operator in $\HSaux$ and
  $\laplX[0] {\wt B}$ the corresponding self-adjoint restriction of the
  Laplacian with domain
  \begin{equation*}
    \dom \laplX[0]B 
    := \set {f_0 \in \HS_0^2}
            {\wt \gamma_1 \de f_0 = \wt B \gamma_0 f_0}
  \end{equation*}
  (Robin type boundary conditions). Assume that $\dom (\laplX[0]B)^*
  \subset \HS_0^1$
  \begin{enumerate}
  \item For $z \notin \spec{\laplD[0]}$ we have $\ker (\laplX[0] B -
    z) = \beta_0^z \Lambda^{-1} \ker (\Lambda(z) - B)$.
  \item
    For $z \notin \spec{\laplX[0] B} \cup \spec{\laplD[0]}$ we have $0
    \notin \spec{\Lambda(z) - \wt B}$ and Krein's formula
    \begin{equation*}
        (\laplD[0] - z)^{-1} - (\laplX[0] B - z)^{-1}
        = \beta_0^z (\Lambda(z)-\wt B)^{-1} (\wt \beta_0^{\conj z})^*
    \end{equation*}
    is valid, where $(\wt \beta_0^{\conj z})^*$ is the adjoint of
    $\map{\beta_0^{\conj z}}{\HSaux^{1/2}}{\HS_0}$ considered as an
    unbounded operator $\partmap {\wt \beta_0^{\conj z}}\HSaux
    {\HS_0}$ with domain $\HSaux^{1/2}$.
  \item We have
    \begin{equation*}
          \spec {\laplX[0] B} \setminus \spec {\laplD[0]}
          = \set{ z \notin \spec {\laplD[0]}} 
                { 0 \in \spec{\Lambda(z) - \wt B}}.
    \end{equation*}
  \end{enumerate}
\end{theorem}
\begin{proof}
  The proof follows from \Thm{krein} because $\Lambda(z) - \wt
  B=\Lambda(Q_0^z - B)$ and $(\wt \beta_0^z)^*= \Lambda(\beta_0^z)^*$.
\end{proof}

%
\section{Boundary triples}
\label{sec:bd.triple}
%
In this section we show how the first order approach of the last
section fits into the setting of boundary triples in the usual sense.
We only sketch the ideas here; for more details on boundary triples,
we refer to~\cite{bgp:pre06,dhms:06} and the references therein.

\begin{definition}
  \label{def:bd.triple}
  Let $\HS$ be a Hilbert space with a closed operator $D$ in $\HS$.
  Assume furthermore that $\wt \HSaux$ is another Hilbert space, and
  $\map{\Gamma_0,\Gamma_1} {\dom D} {\wt \HSaux}$ are two linear maps.  We
  say that $(\wt \HSaux, \Gamma_0, \Gamma_1)$ is an \emph{(ordinary)
    boundary triple for $D$} iff
  \begin{subequations}
    \label{eq:bd.triple}
    \begin{gather}
    \label{eq:bd.triple1}
      \iprod[\HS] {Df} g - \iprod[\HS] f {Dg} 
      =  \iprod[\wt \HSaux] {\Gamma_0 f} {\Gamma_1 g} 
        -\iprod[\wt \HSaux] {\Gamma_1 f} {\Gamma_0 g}, \qquad
        \forall \, f,g \in \dom D\\
    \label{eq:bd.triple2}
      \map{\Gamma_0 \oplussplit \Gamma_1} {\dom D} 
             {\wt \HSaux \oplus \wt \HSaux}, \,
             f \mapsto \Gamma_0 f \oplus \Gamma_1 f
                       \quad \text{is surjective}\\
    \label{eq:bd.triple3}
      \ker (\Gamma_0 \oplussplit \Gamma_1) =
      \ker \Gamma_0 \cap \ker \Gamma_1 \quad \text{is dense in $\HS$.}
    \end{gather}
  \end{subequations}
\end{definition}

\begin{lemma}
  Let $\HS := \HS_0 \oplus \HS_1$ and $(\HS, \HSaux, \gamma_0)$ be a
  first order boundary triple as in \Def{ext.der}. Write
  \begin{equation*}
    D :=
    \begin{pmatrix}
      0 & \ded \\ \de & 0
    \end{pmatrix}, \quad
    \dom D := \HS^1 := \HS_0^1 \oplus \HS_1^1, \quad
    \normsqr[\HS^1] f = \normsqr[\HS] f + \normsqr[\HS]{Df},
  \end{equation*}
  and $\Gamma_p f := \gamma_p f_p$ for $f = f_0 \oplus f_1 \in \HS^1$.
  Then $(\HSaux^{1/2},\Gamma_0,\Gamma_1)$ is an ordinary boundary
  triple for $D$. 
\end{lemma}
\begin{proof}
  The Green's formula~\eqref{eq:bd.triple1} follows from
  \begin{multline*}
      \iprod[\HS] {Df} g - \iprod[\HS] f {Dg} 
      =  \iprod[\HS_1] {\de f_0} {g_1} 
       - \iprod[\HS_0] {f_0} {\ded g_1}      
       + \iprod[\HS_0] {\ded f_1} {g_0} 
        -\iprod[\HS_1] {f_1} {\de g_0} \\
      =  \iprod[\HSaux^{1/2}] {\gamma_0 f_0} {\gamma_1 g_1} 
        -\iprod[\HSaux^{1/2}] {\gamma_1 f_1} {\gamma_0 g_0}
  \end{multline*}
  by \Lem{green}. The second condition~\eqref{eq:bd.triple2} follows
  from $\Gamma_0 \oplussplit \Gamma_1=\gamma_0 \oplus \gamma_1$ and
  the surjectivity of $\map{\gamma_p}{\HS_p^1}{\HSaux^{1/2}}$.  The
  last condition~\eqref{eq:bd.triple3}, i.e., the density of $\HSn^1
  := \HSn_0^1 \oplus \HSn_1^1$ in $\HS$, is a consequence of
  \Defenum{ext.der}{bd}.
\end{proof}

The next lemma can be proved readily:
\begin{lemma}
  \label{lem:iso.psi}
  Set $\mc N^w := \ker (D - w)$. If $w \ne 0$ then $\map {\psi_p^w}
  {\mc N_p^{w^2}}{\mc N^w}$ with
  \begin{equation*}
    \psi_0^w f_0 := \frac 1 {\sqrt 2}
    \begin{pmatrix}
      f_0 \\ \frac 1w \de f_0
    \end{pmatrix},
        \qquad
    \psi_1^w f_1 := \frac 1 {\sqrt 2}
    \begin{pmatrix}
      \frac 1w \ded f_1 \\ f_1
    \end{pmatrix}
  \end{equation*}
  are topological isomorphisms. In particular, for $w=\pm \im$, they
  are unitary.
\end{lemma}

\begin{corollary}
  \label{cor:def.ind} The operator $D$ has zero defect index, i.e., $\mc
  N^{\im}=\ker (D-\im)$ and $\mc N^{-\im}=\ker(D+\im)$ are isomorphic.
\end{corollary}

The next lemma is a well known fact; we give a proof for completeness.
\begin{lemma}
  \label{lem:osum3}
  If $w \ne 0$ then $\HS^1 = \HSn^1 \dplus \mc N^{w} \dplus \mc
  N^{-w}$ (topological direct sum), and the projection $P^w$ onto $\mc
  N^w$ is given by
  \begin{equation*}
    P^w = \frac 12
    \begin{pmatrix}
      P_0^{w^2} & \frac 1w \ded P_1^{w^2}\\
      \frac 1w \de P_0^{w^2} & P_1^{w^2}
    \end{pmatrix}.
  \end{equation*}
  If $w=\pm \im$, then we have $\HS^1 = \HSn^1 \oplus \mc N^{\im}
  \oplus \mc N^{-\im}$ (orthogonal direct sum), and $P^{\pm \im}$ are
  \emph{orthogonal} projections (in $\HS^1$).
\end{lemma}
\begin{proof}
  Recall that $P_p^z$ is the projection onto $\mc
  N_p^z=\ker(\lapl[p]-z)$.  Denote by $\ring P_p := 1 - P_p^z$ the
  projection onto $\HSn_p^1$ and set $\ring P := \ring P_0 \oplus
  \ring P_1$. Then we can decompose $f \in \HS^1$ as
\begin{equation*}
    f = \ring P f + P^w f + P^{-w} f,
  \end{equation*}
  since $P^w + P^{-w} = P_0^{w^2} \oplus P_1^{w^2}$ and $\ring P +
  (P_0^{w^2} \oplus P_1^{w^2})=1$. A simple calculation shows that $D
  P^w = w P^w$, i.e., that $P^w f \in \mc N^w$; in addition,
  $(P^w)^2=P^w$, i.e., $P^w$ is a projection; and $\ring P f \in
  \HSn^1$. 

  \sloppy The sum of eigenspaces associated to different eigenvalues
  is direct, and $\mc N^w \dplus \mc N^{-w} = \mc N_0^{w^2} \oplus \mc
  N_1^{w^2}$ (\Lem{iso.psi}). Since in addition, $\HS_p^1 = \HSn_p^1
  \dplus \mc N_p^{w^2}$, it follows that the sum $\HS^1=\HSn^1 \dplus
  \mc N^w \dplus \mc N^{-w}$ is direct. The direct sum is also
  topological since the projections are bounded operators. The
  orthogonality for $w=\pm \im$ can be checked easily.
\end{proof}

\begin{lemma}
  Let $D^{\min}$ be the restriction of $D$ onto $\dom D^{\min} =
  \HSn^1 := \HSn_0^1 \oplus \HSn_1^1 = \ker(\Gamma_0 \oplussplit
  \Gamma_1)$. Then $(D^{\min})^*=D$.
\end{lemma}
\begin{proof}
  We refer to \cite[Thm.~1.13~(1)$\Rightarrow$(4)]{bgp:pre06} for a
  proof. Note that $D$ has self-adjoint restrictions since the defect
  index is $0$ by \Cor{def.ind}.
\end{proof}

We write $D^\Dir := D \restr {\ker \Gamma_0}$, the \emph{Dirichlet
  Dirac operator}, and $D^\Neu := D \restr {\ker \Gamma_1}$, the
\emph{Neumann Dirichlet operator}. Note that $(D^\Dir)^2 = \laplD[0]
\oplus \laplD[1]$ and $(D^\Neu)^2 = \laplN[0] \oplus \laplN[1]$.
\begin{lemma}
  \label{lem:krein.g3}
  Let $w \notin \spec {D^\Dir}$. The operator
  $\map{\Gamma_0 \restr {\mc N^w}}{\mc N^w} {\HSaux^{1/2}}$ has a
  bounded inverse $\beta^w$, and $w \mapsto \beta^w$ is a
  \emph{$\Gamma$-Krein field}, i.e.,
  \begin{subequations}
    \begin{gather}
      \label{eq:g.krein1}
      \map {\beta^w}{\HSaux^{1/2}}{\mc N^w} \quad\text{is a
        topological isomorphism and}\\
      \label{eq:g.krein2}
      \beta^{w_1} = U^{w_1,w_2} \beta^{w_2}, \qquad w_1,w_2 \notin
      \spec {D^\Dir},
    \end{gather}
  \end{subequations}
  where
  \begin{equation*}
    U^{w_1,w_2} := (D^\Dir - w_2) (D^\Dir - w_1)^{-1} =
    1 + (w_1 - w_2) (D^\Dir - w_1)^{-1}.
  \end{equation*}
  Furthermore, $\beta^w = \sqrt 2 \psi_0^w \beta_0^{w^2}$, where
  $\beta_0^z$ is the Krein $\gamma$-function of order $0$ associated
  with the first order boundary triple $(\HS, \HSaux, \gamma_0)$
  (cf.~\Def{krein.g}) and $\psi_0^w$ is defined in \Lem{iso.psi}.
\end{lemma}
\begin{proof}
  For the proof of the first assertion, we refer again
  to~\cite[Thm.~1.23~(2a--b)]{bgp:pre06}. The relation with
  $\beta_0^{w^2}$ follows from the fact that $\Gamma_0 = \gamma_0
  \pi_0$, where $\map {\pi_0} {\HS^1}{\HS_0^1}$, $f \mapsto f_0$; and
  the inverse of $\sqrt 2 \psi_0^w$ is $\pi_0$ (restricted to the
  appropriate subspaces).
\end{proof}

\begin{lemma}
  \label{lem:krein.q3}
  The operator $\map{Q^w := \Gamma_1
    \beta^w}{\HSaux^{1/2}}{\HSaux^{1/2}}$ defines the Krein
  Q-function $w \mapsto Q^w$, i.e.,
  \begin{equation*}
    Q^{w_1} - (Q^{\conj w_2})^* = 
    (w_1 - w_2) (\beta^{\conj w_2})^* \beta^{w_1} \qquad 
       w_1,w_2 \notin \spec {D^\Dir}.
  \end{equation*}
  Furthermore, $Q^w=\frac 1 w Q_0^{w^2}$, where $Q_0^z$ is the Krein
  Q-function associated to the first order boundary triple $(\HS,
  \HSaux, \gamma_0)$ (cf.~\Def{dn.z}).
\end{lemma}
\begin{proof}
  For the proof of the first assertion, we refer again
  to~\cite[Thm.~1.23~(2c)]{bgp:pre06}. The other follows
  straightforward.
\end{proof}

Further results like Krein's resolvent formula or the spectral
relation for self-adjoint restrictions $D^B$ of $D$ can be found
e.g.~in~\cite{bgp:pre06}. In particular, if $B$ is bounded and
self-adjoint in $\HSaux^{1/2}$ then the restriction of $D$ to
\begin{equation*}
  \dom D^B 
    := \set {f \in \HS^1}
            {\Gamma_1 f = B \Gamma_0 f}
    = \set {f \in \HS^1}
            {\gamma_1 f_1 = B \gamma_0 f_0}
\end{equation*}
defines a self-adjoint operator $D^B$. The Laplacian $(D^B)^2$ acts on
each component as the Laplacian $\lapl[p] f_p$, but with domain
\begin{align*}
  \dom (D^B)^2 
    &= \set {f \in \dom D^B}
            {D f \in \dom D^B}\\
    &= \set {f \in \HS^2}
            {\gamma_1 f_1 = B \gamma_0 f_0, \quad
             \gamma_1 \de f_0 = B \gamma_0 \ded f_1}.
\end{align*}
Note that this domain is different from $\dom \laplX[0] B \oplus \dom
\laplX[1] B$ (cf.~\Def{lapl.sa}) since the two components in
$\dom(D^B)^2$ are coupled.

%
\section{Manifolds with boundary}
\label{sec:example}
%
In this section we present our main example and show how it fits into
the abstract setting of first order boundary triples of
\Sec{1st.order} (see also~\cite{arlinskii:00}).

Let $X$ be a Riemannian manifold with boundary $\bd X$ equipped with
their natural volume measures. Denote the cotangential bundle (or
bundle of $1$-forms) by $T^*X$. The data we need to fix are the
following:
\begin{align*}
  \HS_0 &:= \Lsqr X,     & \HS_0^1&:= \Sob X,\\
  \HS_1 &:= \Lsqr{T^*X}, & \map \de {\Sob X&}{\Lsqr{T^*X}},
\end{align*}
where $\Lsqr X$ and $\Lsqr {T^*X}$ are the spaces of square-integrable
functions and sections over the cotangent ($1$-form) bundle, and where
$\de$ stands for the usual exterior derivative with domain $\HS_0^1 :=
\Sob X$, the Sobolev space of functions $f \in \Lsqr X$ such that
$|\de f| \in \Lsqr X$ (or $\de f \in\Lsqr {T^*X}$, what is the same).

For the boundary map, we need to fix the boundary space $\HSaux :=
\Lsqr {\bd X}$, and we define
\begin{align*}
  \map {\gamma_0}{\Sob X&}{\Lsqr {\bd X}}, &
  \gamma_0 f := f \restr {\bd X}.
\end{align*}
Note that the norm of $\gamma_0$ depends on the local geometry of $X$
near $\bd X$. The range of $\gamma_0$ is $\HSaux^{1/2}= \Sob[1/2]{\bd
  X}$ together with the intrinsic norm defined in \Sec{1st.order},
namely
\begin{equation*}
  \normsqr[{\Sob[1/2]{\bd X}}] \phi :=
  \normsqr[\Sob X] {f_0} = \normsqr[\Lsqr X] {f_0} 
              + \normsqr[\Lsqr X]{\de f_0},
\end{equation*}
where $f_0$ is the solution of the Dirichlet problem $(\lapl[0]+1)f_0=0$
and $\gamma_0 f_0 = \phi$.  Since $\Sob[1/2]{\bd X} \ne \Lsqr{\bd X}$,
the boundary map $\gamma_0$ is proper.

After defining these data, we obtain $\HSn_0^1 = \Sobn X = \ker
\gamma_0$ and $\de_0 := \de \restr {\Sobn X}$.  Furthermore, $\ded =
\de_0^*$ is the divergence operator.  Comparing the abstract Green's
formula in \Lem{green} with Green's formula
\begin{equation*}
  \int_X \iprod {\de f} \eta{}_x \dd x 
  - \int_X \conj f \, \ded \eta \dd x 
  = \int_{\bd X} (\conj f  \, \eta_{\mathrm n})\restr{\bd X},
\end{equation*}
where $\eta_{\mathrm n}$ stands for the normal component of the
$1$-form $\eta$ near $\bd X$, we see that
\begin{equation*}
  \wt \gamma_1 \eta = \eta_{\mathrm n} \restr{\bd X}
\end{equation*}
  
\begin{remark}
  \label{rem:not.sob}
  Note that $\HS_1^1 := \dom \ded \subset \Lsqr {T^*X}$ is \emph{not}
  the Sobolev space of order $1$ on $1$-forms, defined locally via
  charts. Therefore, $\map{\wt \gamma_1}{\dom \ded} {\Sob[-1/2]{\bd
      X}}$, and $\wt \gamma_1$ does not map into $\Sob[1/2]{\bd X}$,
  as one could naively guess.
\end{remark}
The Dirichlet-to-Neumann map in this case is
\begin{equation}
  \label{eq:dn}
  \Lambda(z) \phi = \normder h_0, \qquad\text{where}\qquad
  \lapl[0] h_0 = z h_0,
  \qquad h_0 \restr {\bd X} = \phi
\end{equation}
for $\phi \in \Sob[1/2]{\bd X}$ and $z \notin \spec {\laplD[0]}$
(cf.~\Def{dn.z}).

Self-adjoint boundary conditions of the Laplacian on $0$-forms like
\emph{Robin boundary conditions} are now given as follows: Let $\wt B$
be a bounded, real-valued function on $\bd X$ and set $B:=
\Lambda^{-1}\wt B$. Then $B$ is bounded and self-adjoint on
$\HSaux^{1/2}$ (\Lem{b.12}) and
\begin{equation*}
  \dom (\laplX[0]B)^* = 
    \set{f_0 \in  \laplX[0]{\max} }
     {\normder f_0 \restr{\bd X} = \wt B f_0 \restr{\bd X}}
\end{equation*}
is indeed a subset of the \emph{Sobolev} space $\Sob[2] X$ (see
e.g.~\cite[Prop.~III.5.2]{grubb:68}
or~\cite[Thm.~7.4]{lions-magenes:72}). In particular, the domain
condition $\dom (\laplX[0]B)^* \subset \HS_0^1=\Sob X$ is fulfilled,
and the above domain defines a \emph{self-adjoint} Laplace operator
(cf.~\Lem{lapl.sa}).

Note that in general, the Robin boundary conditions cannot be
expressed as $(D^B)^2$ where $D^B$ is a self-adjoint restriction of
the Dirac operator (cf.~the end of \Sec{bd.triple}). This is another
justification of our first order approach (instead of directly
starting from an ordinary boundary triple as in \Sec{bd.triple}).
  
\begin{remark}
  \label{rem:lapl.max.mfd}
  The first order approach to boundary triples enables us to use the
  \emph{natural} boundary maps $\gamma_0 f= f \restr {\bd X}$ and $\wt
  \gamma_1 \eta = \eta_{\mathrm n} \restr {\bd X}$, in contrast to the
  second order approach using the Laplacian as
  e.g.~in~\cite{bmnw:pre07,posilicano:pre07}. In the second order
  approach, the maximal domain of the Laplacian
  \begin{equation*}
    \dom \lapl^{\max} = \set{f \in \Lsqr X} 
    { \lapl f \in \Lsqr X}
  \end{equation*}
  is \emph{not} a subset of the Sobolev space $\Sob X$. In particular,
  $f \restr {\bd X}$ is not in $\Lsqr {\bd X}$, but only in
  $\Sob[-1/2]{\bd X}$; and $\normder f \restr{\bd X} \in
  \Sob[-3/2]{\bd X}$ (see e.g.~\cite{grubb:68, grubb:06,
    lions-magenes:72}). In particular, Green's formula
  (cf.~\eqref{eq:bd.triple1}) fails to hold with the natural boundary
  maps.
\end{remark}

\providecommand{\bysame}{\leavevmode\hbox to3em{\hrulefill}\thinspace}
\renewcommand{\MR}[1]{}

\end{document}